\documentclass[12pt,psfig]{article}
\usepackage{graphics,epsfig,amssymb}
\pdfoutput=1
\usepackage{cite}

\usepackage{amsmath}
\usepackage{amsthm}
\usepackage{amsfonts}
\usepackage{amssymb,latexsym}
\usepackage[all]{xy}
\usepackage{graphicx}
\usepackage[mathscr]{eucal}
\usepackage{verbatim}
\usepackage{hyperref}

\usepackage{txfonts}

\usepackage{mathrsfs}
\usepackage{color}
\usepackage{hyperref}

\theoremstyle{plain}
    \newtheorem{thm}{Theorem}[section]
    \newtheorem{prop}{Proposition}[section]
    \newtheorem{lemma}{Lemma}[section]

\numberwithin{equation}{section}

\usepackage{graphicx}
\usepackage{pict2e}
\usepackage{slashbox}

\begin{document}

\title{Uniqueness, symmetry and convergence  of   positive ground state solutions  of   the Choquard   type equation on a  ball \thanks{The first author was partially supported by the  Natural Science Foundation of Hunan Province (Grant No.
2020JJ5151) , the National Natural Science Foundation of China (Grant No. 12001188) and  Scientific Research Fund of Hunan Provincial Education Department (Grant No. 19C0781); the second author  was partially supported by the
National Natural Science Foundation of P. R. China ( No.
11971494).}}

\author{ 
Hui Guo\\
  College of Mathematics and Computing Science,\\
Hunan University of Science and Technology, \\
Xiangtan, Hunan 411201, P. R. China\\
Email: huiguo\_math@163.com;\\
Tao Wang\\
  College of Mathematics and Computing Science,\\
Hunan University of Science and Technology, \\
Xiangtan, Hunan 411201, P. R. China\\
Email: wt\_61003@163.com;\\
  Taishan Yi \\
  School of Mathematics (Zhuhai),\\
Sun Yat-sen University,\\
Zhuhai, Guangdong 519082,  P. R. China\\
Email:yitaishan@mail.sysu.edu.cn
\\}
\date {}
\maketitle

\begin{abstract}
This paper is  concerned with the qualitative properties of  the positive ground state solutions  to the nonlocal  Choquard type equation on a ball $B_R$.  First,  we prove the  radial symmetry of  the positive ground state solutions by using Talenti's inequality.  Next we develop Newton's Theorem and then resort to the contraction mapping principle to  establish  the  uniqueness of  the positive ground state solutions. Finally, by constructing cut-off functions and applying energy comparison method, we show the convergence of  the   positive ground state solutions  as $R\to \infty$.   Our results generalize  and improve the existing ones in the  literature.
 \end{abstract}

\section{Introduction}
The interesting  stationary Choquard    equation in the whole space $\mathbb{R}^N$ with   $N\geq3$,
\begin{equation}\label{model1}
-\Delta u+u=(\int_{\mathbb{R}^N}\frac{|u(y)|^2}{|x-y|^{N-2}}dy)u, \mbox{\ in\ } \mathbb{R}^N
\end{equation}
which  is usually called
Schr\"{o}dinger-Newton equation,  has several physical origins. When $N=3$, the nonlocal  problem \eqref{model1}
 becomes the classical stationary Choquard-Pekar
equation
 \begin{equation}\label{model2}
-\Delta
u+u=(\int_{\mathbb{R}^3}\frac{|u(y)|^2}{|x-y|}dy)u,\
 \mbox{\ in\ } \mathbb{R}^3.
\end{equation}
Equation \eqref{model2} first appeared at least as early as in
1954, in a work by Pekar describing the quantum mechanics of polaron
at rest. In 1976, Choquard used \eqref{model2} to model an
electron trapped in its own hole, in a certain  approximation to
Hartree-Fock theory of one-component plasma \cite{LEH}. This
equation also arises in many interesting situations related to the
quantum theory of large systems of nonrelativistic bosonic atoms and
molecules, see \cite{Lieb-Simon}.
In mathematical contents, the existence and qualitative properties of solutions for \eqref{model1} have been studied widely and intensively in the literature. See \cite{ccs,gmv,gmvj,LEH,LPL,vv} and references therein for the existence of ground state  solutions, multiple  solutions  and nodal solutions to \eqref{model1}. About the qualitative properties such as regularity, symmetry, uniqueness and decay of nontrivial solutions of \eqref{model1}, we can refer to \cite{cp,vv,MZ,rd,wy,wy1,xcl} for instance.

As we know,  by  rescaling, \eqref{model1} is equivalent to
\begin{equation}\label{model3}
\left\{
\begin{array}{lll}
-\Delta u+u=wu,\ \mbox{in }\mathbb{R}^N,\\
-\Delta w=|u|^2,\ \mbox{in }\mathbb{R}^N.
\end{array}\right.
\end{equation}
Let $B_R\subset \mathbb{R}^N$ be an open  ball domain with radius $R>0$ centered at the origin. Then the Dirichlet problem of \eqref{model3}  on  $B_R$ is
\begin{equation}\label{model4}
\left\{
\begin{array}{lll}
-\Delta u+u=wu\ \mbox{in } B_R,\\
-\Delta w=|u|^2\ \mbox{in } B_R,\\
w=u=0\  \mbox{in }  \partial B_R.
\end{array}\right.
\end{equation}
By using Green's function, \eqref{model4}
can  be rewritten as the following Choquard type equation on a ball
\begin{equation}\label{model5}
-\Delta u+u=\left(\int_{B_R}G(x,y)|u^2(y)|dy\right)u,\ \ x\in B_R.
\end{equation}
We point out that
\begin{displaymath}
G(x,y)=\frac{1}{|y-x|^{N-2}}-\frac{1}{(\frac{|x|}{R}|y-\tilde{x}|)^{N-2}},\ (x,y\in B_R \mbox{ with } x\neq y),
\end{displaymath}
where $\tilde{x}=\frac{R^2x}{|x|^2}$ is  the dual point of  $x$ with respect to $\partial B_R$. Clearly, $G(x,y)>0$, $G(x,y)=G(y,x)$  for $x,y\in B_R$ and $x\neq y$.

If  a function $u$ is said to be a ground state solution of  \eqref{model1} (or \eqref{model5}), if $u$ solves \eqref{model1} (or \eqref{model5}) and
minimizes the energy functional associated with \eqref{model1} (or \eqref{model5})  among all possible nontrivial solutions. In the whole space $\mathbb{R}^N$, Moroz and
Van Schaftingen \cite{vv} obtained   the regularity, radial
symmetry and asymptotic behavior of    positive ground state solutions of  the Choquard equation \eqref{model1}.  This combined with the uniqueness of positive radial solutions of  \eqref{model1} (see \cite{wy}), implies that the positive ground state solutions of   \eqref{model1} are uniquely determined, up to translations.
On a ball $B_R$,  Wang and Yi   \cite{wy1} proposed that the positive ground state solutions of \eqref{model5} may be only  axially symmetric by using its separability property.
 When $N=3$,  Feliciangeli  and  Seiringer \cite{fs} proved the radial symmetry and uniqueness of   positive solutions with a  prescribed $L^2-$ norm for the Choquard  type equation \eqref{model5}.
Motivated by the above work,   one natural question is to ask  whether the positive ground state solution of  \eqref{model5} on a ball $B_R$ is radially symmetric and unique  when $N\geq3$. We know  that for the classical  local elliptic equations
\begin{equation}\label{model40}
-\Delta u+u=u^p  \ \mbox{\ in } B_R,
\end{equation}
with $p>1$, Gidas,  Ni and  Nirenberg  \cite{gn} showed  us that all the positive ground state solution are radially symmetric on a ball $B_R$. Later,
Kwong \cite{kwk}  established the uniqueness of the positive, radially symmetric solution to \eqref{model40}. This yields that the positive ground state solution to \eqref{model40} is unique.  For more results on  the uniqueness  for the general local elliptic equations, we can refer to \cite{cvc,ms} and references therein.   However,  compared with the problem on the  uniqueness for the  local elliptic equations, the appearance of  convolution term makes the nonlocal  Choquard equations extremely difficult to handle. In order to obtain our main results,   we first prove the radial symmetry of all the positive ground state solutions of \eqref{model5}.  Then we resolve the convolution term by  developing Newton's Theorem, which  allows us to shift the uniqueness study to the ordinary differential equation.
Finally, we investigate the asymptotic behavior of the unique  positive ground state solution of  \eqref{model5}.

The remainder of this paper is organized as follows. In Section 2, some notations and  preliminary results are presented.   Section 3 is  devoted to
the proof  of the radial symmetry of  all the positive  ground state solutions   to    \eqref{model5}  with $N\geq3$  by using Talenti's inequality.
In Section 4,  we develop Newton's Theorem and then take  advantage of the  contraction mapping principle to  establish  the  uniqueness of  the positive ground state solutions to    \eqref{model5} when $N=3,4,5,6$.  In Section 5, by constructing cut-off functions and applying energy comparison method, we show  that the unique  positive ground state solution of     \eqref{model5} converges to  the unique  positive ground state solution of     \eqref{model1} as $R\to \infty$.

\section{Preliminaries}

In this section, some notations are collected as follows.

  \noindent $\bullet$ Let $N\in \mathbb{N}$ and $H^1(\mathbb{R}^N)$ is the Sobolev space with standard norm
$\|u\|=(\int_{\mathbb{R}^N}|\nabla u|^2+|u|^2)^{\frac{1}{2}}$.
We can identify $u\in H_0^1(\Omega)$ with its extensions to $\mathbb{R}^N$ obtained by setting $u=0$ in $\mathbb{R}^N\backslash \Omega.$\\
$\bullet$  The dual space of  $H^{1}(\mathbb{R}^N)$ is denoted by $H^{-1}(\mathbb{R}^N)$.\\
$\bullet$ Let $\langle\cdot,\cdot\rangle$ be the duality pairing between $H^{1}(\mathbb{R}^N)$ and $H^{-1}(\mathbb{R}^N).$\\
\noindent$\bullet$ Let  $\mathbb{S}_{r,x}^{N-1}$ be
the sphere with radius $r>0$ centered at the point $x$ in $\mathbb{R}^N$.  $|\mathbb{S}_{r,x}^{N-1}|$  denotes its $N-1$ dimensional area. For simplicity, $\mathbb{S}^{N-1}$ and $|\mathbb{S}^{N-1}|$ denote the sphere of unit radius centered at the origin and its $N$ dimensional area, respectively.\\
 $\bullet$ For $1\leq s<\infty$, $L^s(\Omega)$ denotes the Lebesgue space with the norm
  $\|u\|_{L^s(\Omega)}=\left(\int_{\Omega} |u|^sdx \right)^{\frac{1}{s}}.$\\
  $\bullet$  $C$ may  represent  different  positive constants.

As usual, for $N\geq3$,  the corresponding energy functional  $I_R:H^1_0(B_R)\to \mathbb{R}$ associated to \eqref{model5} is
\begin{equation}\label{18}
I_R(u)=\frac{1}{2}\int_{B_R}(|\nabla u|^2+|u|^2)dx-\frac{1}{4}\int_{B_R}\int_{B_R}G(x,y)|u(y)|^2|u(x)|^2 dxdy,
\end{equation}
due to the symmetry and positivity of $G(x,y)$ for $x,y\in B_R$ and $x\neq y$.
 It is easy to check that  $I_R$ is $C^1-$functional
and its Gateaux  derivative is given by
$$\langle I_R'(u),v\rangle=\int_{B_R}(\nabla u\nabla v+uv)dx-\int_{B_R}\int_{B_R}G(x,y)|u(y)|^2u(x)v(x)dxdy$$
for any $v\in H^1_0(B_R).$
Recall that the critical points of $I_R$ are  solutions of \eqref{model5} in the weak sense.
Let $$c_R:=\inf\limits_{u\in\mathcal{N}_R}I_R(u),$$ where the Nehari manifold
$$\mathcal{N}_R=\{u\in H_0^1(B_R)\backslash{\{0\}}:\langle I_R'(u),u\rangle=0\}.$$

For the sake of convenience,
$$\mathbb{D}(u):=\int_{B_R}\int_{B_R}G(x,y)|u(y)|^2|u(x)|^2dxdy.$$
The proof of the following properties of the Nehari manifold $\mathcal{N}_R$  is standard and hence is omitted here.
\begin{lemma}\label{lemma7}
The following statements are true:
\begin {itemize}
\item[\rm(i)] $0\notin \partial\mathcal{N}_R$ and $c_R>0$;
\item[\rm(ii)] For any \,$u\in H^1_0(B_R)\backslash\{0\}$\,,
there exists a unique  $t_u\in(0,\infty)$ such that $t_uu\in \mathcal{N}_R$ and
 $t_u=\left(\frac{\|u\|^2}{\mathbb{D}(u)}\right)^\frac{1}{2}.$
Furthermore,
\begin{equation}\label{19}
I_R(t_u u)=\sup\limits_{t>0}I_R(tu)=\frac{1}{4}\frac{\|u\|^4}{\mathbb{D}(u)};
\end{equation}
\item[\rm(iii)] $c_R=\inf\limits_{u\in\mathcal{N}_R} I_R(u)=\inf\limits_{u\in H^1_0(B_R)\backslash\{0\}}\sup\limits_{t>0}I_R(tu)$.
\end{itemize}
\end{lemma}

By using  standard Nehari  manifold method,  we can obtain the existence of ground state solutions  of  \eqref{model5} in $H^1_0(B_R)$ (see \cite{wy1}).
In the following parts, we always assume that   $\phi_R\in H^1_0(B_R)$ is a  ground state solution of \eqref{model5}.

\section{Positivity and  symmetry}

In this section, we shall prove the positivity and radial symmetry of ground state solutions  of  \eqref{model5}.  In order to achieve it, we need the following Talenti's inequality.
\begin{lemma}\label{lemma6}(see \cite[Theorem 2.4]{fs})
Let $0\leq f\in L^2(B_R)$, and let $u,v\in H^1_0(B_R)$ solve
\begin{equation*}\left\{
\begin{array}{lll}
-\Delta u=f, \ x\in B_R,\\
u=0,\ x\in \partial B_R,
\end{array}\right.
\end{equation*}
and
\begin{equation*}\left\{
\begin{array}{lll}
-\Delta v=f^*, \ x\in B_R,\\
v=0,\ x\in \partial B_R.
\end{array}\right.
\end{equation*}
Then $u^*\leq v$ a.e. in $B_R$. If additionally, $u^*(x_0)=v(x_0)$ for
some $x_0$ with $|x_0|=t\in (0,R),$ then $u(x)=v(x)$ and $f(x)=f^*(x)$ for
all $x$ with $t\leq|x|\leq R.$ Here $u^*$ and $f^*$ are the symmetric decreasing rearrangement of $u$ and $f$, respectively.
\end{lemma}

Based on Lemma \ref{lemma6}, we  can obtain  the following result.
\begin{prop}\label{proposition1}
  $\phi_R \in {C}^2(B_R)\bigcap
H^1_0(B_R)$  is strictly positive, radially symmetric and decreasing.
\end{prop}
\begin{proof}
Since $|\nabla |\phi_R||^2=|\nabla (\phi_R^2)^{\frac{1}{2}}|^2$,    direct calculations imply that  $I_R(\phi_R)=I_R(|\phi_R|)$. This combined with Euler-Lagrange multiplication principle and the definition of ground state solution, yields that $|\phi_R|$ is also a ground state solution of \eqref{model5}. Furthermore, applying
 the standard  elliptic regularity estimates and the strong maximum principle, we deduce that   $\phi_R $ belongs to ${C}^2(B_R)$, and  either $\phi_R>$ or $\phi_R<0$  in $B_R.$
Without loss of generality, we can assume that $\phi_R>0$  in $B_R.$ Let $\phi_R^*$ be the symmetric decreasing rearrangement of $\phi_R.$

 Define
$
u(x)=\int_{B_R}G(x,y)|\phi_R|^2dy, \ v(x)=\int_{B_R}G(x,y)|\phi^*_R|^2dy.
$
According to Lemma \ref{lemma6}, we  deduce that $u^*\leq v$. Then, by using symmetric rearrangement
inequalities, we have
\begin{equation}\label{model38}
\int_{B_R}|\nabla \phi^*_R|^2dx\leq \int_{B_R}|\nabla \phi_R|^2dx,\ \int_{B_R}| \phi^*_R|^2dx=\int_{B_R}|\phi_R|^2dx
\end{equation} and
\begin{equation}\label{model17}
\begin{split}
\int_{B_R}\int_{B_R}G(x,y)|\phi^*_R(y)|^2|\phi^*_R(x)|^2dxdy&=\int_{B_R}v(x)|\phi^*_R(x)|^2dx\\
&\geq\int_{B_R}u^*(x)|\phi^*_R(x)|^2dx\\
&\geq\int_{B_R}u(x)|\phi_R(x)|^2dx\\
&=\int_{B_R}\int_{B_R}G(x,y)|\phi_R(y)|^2|\phi_R(x)|^2dxdy.
\end{split}
\end{equation}
This combined with \eqref{model17},  implies that
$I_R(\phi_R)\geq I_R(\phi_R^*)$.

In what follows,  we shall show  that
$\phi_R^*$ is a positive ground state solution of \eqref{model5}. In fact, let $F(u)=\langle I'_R(u),u\rangle$.  Then  we have $F(\phi_R^*)\leq F(\phi_R)=0$.  We claim that
$F(\phi_R^*)=0$, otherwise, $F(\phi_R^*)<0$.  This and  Lemma \ref{lemma7} show  that  there exists $t_{\phi_R^*}\in(0,1)$
 such that $t_{\phi_R^*}\phi_R^*\in \mathcal{N}_R,$  and
 $$c_R\leq I_R(t_{\phi_R^*}\phi_R^*)=\frac{1}{4}\frac{\|\phi_R^*\|^4}{\mathbb{D}(\phi_R^*)}
 <\frac{1}{4}\frac{\|\phi_R\|^4}{\mathbb{D}(\phi_R)}=I_R(\phi_R)=c_R.$$
This is a contradiction. Thus the claim holds, that is $\phi_R^*\in \mathcal{N}_R$. So $I_R(\phi_R)\leq I_R(\phi_R^*)$. By \eqref{model38}, we deduce that $I_R(\phi_R)=I_R(\phi_R^*)=c_R.$
This combined with \eqref{model17}, implies that
$$\int_{B_R}v(x)|\phi^*_R(x)|^2dx
=\int_{B_R}u^*(x)|\phi^*_R(x)|^2dx.$$
Since  $v\geq u^*$,  we conclude that $v=u^*$ on $B_R$. Thus, by  Lemma \ref{lemma6},  $u=v$ and then
$\phi_R=\phi_R^*$, which shows that  $\phi_R$  is strictly positive, radially symmetric decreasing. The proof is completed.
\end{proof}

Therefore, we immediately obtain the following symmetry result.
\begin{thm}\label{theorem1}
Assume that $N\geq3$. Then all the   positive  ground state solutions  of  \eqref{model5} are radially symmetric and decreasing.
\end{thm}

\section{Uniqueness}
Based on    the radial symmetry of    the positive ground state solution   of  \eqref{model5} in  Section 3, we shall prove  its uniqueness in this section.  First, we shall develop the Newton's Theorem by listing  the following two lemmas whose proofs are similar as \cite[Theorem 9.7]{Lieb-Loss} with some necessary modifications.
\begin{lemma}\label{lemma1}
Let $N\geq3$ and $$J(r,x)=|\mathbb{S}^{N-1}|^{-1}\int_{\mathbb{S}^{N-1}}|rz-x|^{2-N}dz$$
with  $x\in \mathbb{R}^N$.  Then $J(r,x)=\min\{r^{2-N},|x|^{2-N}\}.$
\end{lemma}
\begin{proof}
First we see that $J(r,x)$ is radial with respect to $x$.  Define $g(x)=|x|^{2-N}$. Then $g$ is a  harmonic function if $x\neq 0.$ Note that
\begin{equation*}
J(r,x)=|\mathbb{S}_{r,-x}^{N-1}|^{-1}\int_{\mathbb{S}_{r,-x}^{N-1}}|y|^{2-N}dy
\end{equation*}
If  $r< |x|$, we see that $|y|^{2-N}$ is a harmonic function in $B_{r,-x}^{N-1}$, where $B_{r,-x}^{N-1}$ is  an open  ball domain with radius $r>0$ centered at the point $-x$.  Thus
\begin{equation}\label{model15}
J(r,x)=g(-x)=|x|^{2-N}.
\end{equation}
If $r>|x|$,  we have
 \begin{equation}\label{model16}
\begin{split}
J(r,x)&=|\mathbb{S}^{N-1}|^{-1}\int_{\mathbb{S}^{N-1}}\Big[|\mathbb{S}^{N-1}|^{-1}\int_{\mathbb{S}^{N-1}}|rz-x|^{2-N}dz\Big]d\xi\\
&=|\mathbb{S}^{N-1}|^{-1}\int_{\mathbb{S}^{N-1}}\Big[|\mathbb{S}^{N-1}|^{-1}\int_{\mathbb{S}^{N-1}}|rz-|x|\xi|^{2-N}dz\Big]d\xi\\
&=|\mathbb{S}^{N-1}|^{-1}\int_{\mathbb{S}^{N-1}}\Big[|\mathbb{S}^{N-1}|^{-1}\int_{\mathbb{S}^{N-1}}|rz-
|x|\xi|^{2-N}d\xi\Big]dz\\
&=|\mathbb{S}^{N-1}|^{-1}\int_{\mathbb{S}^{N-1}}r^{2-N}dz\\
&=r^{2-N}.
\end{split}
\end{equation}

It suffices to prove that when $r=|x|$, the conclusion holds.  Indeed, we claim that when $|x|=r,$
$$\lim\limits_{\tilde{\varepsilon}\to0}\int_{\mathbb{S}_{r,-x}^{N-1}\bigcap B_{\tilde{\varepsilon}}}|y|^{2-N}dy=0.$$
Here  $B_{\tilde{\varepsilon}}\subset \mathbb{R}^N$ is an open  ball domain with radius $\tilde{\varepsilon}>0$ centered at the origin.  Indeed, by direct calculations, we can deduce that
$$\mathbb{S}_{r,-x}^{N-1}\bigcap \partial B_{\varepsilon}=\big\{\frac{\varepsilon^2}{2r^2}x+x^\bot:x^\bot\in(\mbox{ span }\{x\})^\bot \mbox{ and }
|x^\bot|^2=\varepsilon^2-\frac{\varepsilon^4}{4r^2}\big\}, \ (0<\varepsilon<\frac{r}{2}),$$
where  $x^\bot$ is a  orthogonal vector of $x.$
Then\begin{equation*}
\begin{split}
\lim\limits_{\tilde{\varepsilon}\to0}\int_{\mathbb{S}_{r,-x}^{N-1}\bigcap B_{\tilde{\varepsilon}}}|y|^{2-N}dy=&
\lim\limits_{\tilde{\varepsilon}\to0}\int_0^{\tilde{\varepsilon}}\frac{1}{|\varepsilon|^{N-2}}|\mathbb{S}^{N-2}||\varepsilon^2-\frac{\varepsilon^4}{4r^2}|^{\frac{N-2}{2}}d\varepsilon\\
=&\lim\limits_{\tilde{\varepsilon}\to0}\int_0^{\tilde{\varepsilon}}|\mathbb{S}^{N-2}||\frac{4r^2-\varepsilon^2}{4r^2}|^{\frac{N-2}{2}}d\varepsilon\\
=&0.
\end{split}
\end{equation*}
So the claim holds and $J(r,x)$ is well defined in $x\in\mathbb{R}^N$. Furthermore, $J(r,x)$ is continuous with respect to $x.$
So  we obtain $J(r,x)=r^{2-N}.$ This together with   \eqref{model15} and  \eqref{model16}, yields the proof.
\end{proof}

\begin{lemma}\label{lemma8}
Assume that $\varphi\in H^1_0(B_R)$ is a positive radial function. Then for any $x\in B_R,$
\begin{equation}\label{model8}
\begin{array}{lll}
V_{\varphi}(x)&:=&\displaystyle\int_{B_R}\frac{|\varphi(y)|^2}{|y-x|^{N-2}}dy-\int_{B_R}\frac{|\varphi(y)|^2}{\Big|\frac{|y|}{R}x-\frac{R}{|y|}y\Big|^{N-2}}dy\\
&=&\displaystyle
\int_{|y|\leq|x|}(|x|^{2-N}-|y|^{2-N})|\varphi(y)|^2dy+\int_{B_R}
|y|^{2-N}|\varphi(y)|^2dy-R^{2-N}\int_{B_R}|\varphi(y)|^2dy.
\end{array}
\end{equation}
\end{lemma}
\begin{proof}
Let $$P(x):=\int_{B_R}\frac{|\varphi(y)|^2}{|y-x|^{N-2}}dy,\ \mbox{for all $x\in B_R$}.$$
 Then $p$ is a radial function, that is,  $P(x)=P(|x|z)$  for all  $z\in \mathbb{S}^{N-1}.$ By using Lemma \ref{lemma1},
we have
\begin{equation}\label{model10}
\begin{array}{lll}
P(x)&=&\displaystyle |\mathbb{S}^{N-1}|^{-1}\int_{\mathbb{S}^{N-1}}P(x)dz\\
&=&\displaystyle |\mathbb{S}^{N-1}|^{-1}\int_{\mathbb{S}^{N-1}}P(|x|z)dz\\
&=&\displaystyle |\mathbb{S}^{N-1}|^{-1}\int_{\mathbb{S}^{N-1}}\Big[\int_{B_R}\frac{|\varphi(y)|^2}{\Big|y-|x|z\Big|^{N-2}}dy\Big]dz\\
&=&\displaystyle\int_{B_R}\Big[ |\mathbb{S}^{N-1}|^{-1}\int_{\mathbb{S}^{N-1}}\Big||x|z-y\Big|^{2-N} dz\Big]|\varphi(y)|^2 dy\\
&=&\displaystyle|x|^{2-N}\int_{|y|\leq |x|}|\varphi(y)|^2dy+
\int_{|y|>|x|}|y|^{2-N}|\varphi(y)|^2dy\\
&=&\displaystyle\int_{|y|\leq|x|}(|x|^{2-N}-|y|^{2-N})|\varphi(y)|^2dy+\int_{B_R}
|y|^{2-N}|\varphi(y)|^2dy.
\end{array}
\end{equation}

Observe that
\begin{equation}
W(x):=\int_{B_R}\frac{|\varphi(y)|^2}{\Big|\frac{|y|}{R}x-\frac{R}{|y|}y\Big|^{N-2}}dy=
\int_{B_R}\frac{(\frac{R}{|y|})^{N-2}|\varphi(y)|^2}{\Big|x-\frac{R^2}{|y|^2}y\Big|^{N-2}}dy,\ \mbox{for all $x\in B_R$}.
\end{equation}
Clearly, $W(x)$ is radial in $B_R$ and $W(x)=W(|x|z)$ with $z\in \mathbb{S}^{N-1}$. Then
\begin{equation}\label{model9}
\begin{array}{lll}
\displaystyle W(x)&=&\displaystyle|\mathbb{S}^{N-1}|^{-1}\int_{\mathbb{S}^{N-1}} W(|x|z)dz\\
&=&\displaystyle|\mathbb{S}^{N-1}|^{-1}\int_{\mathbb{S}^{N-1}}\Big[\int_{B_R}\frac{(\frac{R}{|y|})^{N-2}|\varphi(y)|^2}
{\Big||x|z-\frac{R^2}{|y|^2}y\Big|^{N-2}}dy\Big]dz\\
&=&\displaystyle\int_{B_R}\Big[|\mathbb{S}^{N-1}|^{-1}\int_{\mathbb{S}^{N-1}}\frac{(\frac{R}{|y|})^{N-2}|\varphi(y)|^2}
{\Big||x|z-\frac{R^2}{|y|^2}y\Big|^{N-2}}dz\Big]dy.
\end{array}
\end{equation}
According to   Lemma \ref{lemma1} again, we may  conclude  that
\begin{equation*}
|\mathbb{S}^{N-1}|^{-1}\int_{\mathbb{S}^{N-1}}
{\Big||x|z-\frac{R^2}{|y|^2}y\Big|^{2-N}}dz=\left\{\begin{array}{lll}
(\frac{R^2}{|y|})^{2-N},\ \mbox{if $|y|\leq \frac{R^2}{|x|}$},\\
|x|^{2-N},\ \mbox{if $|y|> \frac{R^2}{|x|}$.}
\end{array}\right.
\end{equation*}
This combined with \eqref{model9}, yields that for any $x\in B_R$, there holds that $|y|\leq\frac{R^2}{|x|}.$ Hence,
\begin{equation}\label{model24}
\begin{array}{lll}
\displaystyle W(x)&=&\displaystyle\int_{|x|\leq \frac{R^2}{|y|^2}|y|}\Big(\frac{R^2}{|y|^2}|y|\Big)^{2-N}\Big(\frac{R}{|y|}\Big)^{N-2}|\varphi(y)|^2dy
+\int_{|x|>\frac{R^2}{|y|^2}|y|}|x|^{2-N}\Big(\frac{R}{|y|}\Big)^{N-2}|\varphi(y)|^2dy\\
&=&\displaystyle\int_{|x|\leq \frac{R^2}{|y|}}\Big(\frac{R^2}{|y|}\Big)^{2-N}\Big(\frac{R}{|y|}\Big)^{N-2}|\varphi(y)|^2dy
+\int_{|x|>\frac{R^2}{|y|}}|x|^{2-N}\Big(\frac{R}{|y|}\Big)^{N-2}|\varphi(y)|^2dy\\
&=&\displaystyle R^{2-N}\int_{B_R}|\varphi(y)|^2dy.
\end{array}
\end{equation}
Thus the conclusion follows from \eqref{model10} and \eqref{model24}.
\end{proof}

In the sequel, we always assume that $\varphi\in  {C}^2(B_R)\bigcap H^1_0(B_R)$ is a positive radial solution of  \eqref{model5}. Let
\begin{equation}\label{model39}
U_{\varphi}(x):=\int_{|y|\leq|x|}(|y|^{2-N}-|x|^{2-N})|\varphi(y)|^2dy,\ \mbox{for all $x\in B_R$}.
\end{equation}
By applying Lemma \ref{lemma8},  the nonlocal problem
\eqref{model5} becomes
\begin{equation}\label{model13}
(-\Delta+U_\varphi(x))\varphi(x)=\Big(\int_{B_R}|y|^{2-N}|\varphi(y)|^2dy-R^{2-N}\int_{B_R}|\varphi(y)|^2dy-1\Big)\varphi(x).
\end{equation}
We denote by
\begin{equation}\label{model12}
\begin{array}{lll}
\displaystyle\lambda(\varphi)&=&\displaystyle\int_{B_R}|y|^{2-N}|\varphi(y)|^2dy-R^{2-N}\int_{B_R}|\varphi(y)|^2dy-1\\
&=&\displaystyle\int_0^R |\mathbb{S}^{N-1}|s|\varphi(s)|^2ds-R^{2-N}\int_{B_R}|\varphi(y)|^2dy-1.
\end{array}
\end{equation}
Clearly,  $\lambda(\varphi)\in(0,\infty)$  follows from  \eqref{model39} and \eqref{model13}.
Furthermore, we have
\begin{lemma}\label{lemma10}
 $U_\varphi(x)$  is radially symmetric  and bounded.
\end{lemma}
\begin{proof}
By \eqref{model39} and the radial symmetry of  $\varphi$, we have
$$U_{\varphi}(x)=\int_0^r |\mathbb{S}^{N-1}|s^{N-1}(s^{2-N}-r^{2-N})|\varphi(s)|^2ds.$$
This yields the proof.
\end{proof}

 Let $R^*:=\sqrt{\lambda (\varphi)}R.$ The following lemma   illustrates that every radial  solution of \eqref{model13} can obey  the    following  canonical form
\begin{equation}\label{model14}
(-\Delta +U_\phi( x))\phi(x)=\phi(x), \  \phi\in H_0^1(B_{R^*}),
\end{equation} after a suitable scaling.  This  is critical  in the proof of  the  uniqueness of   positive radial solutions of \eqref{model5}.
\begin{lemma}\label{lemma2}
Let $\varphi_{\lambda}(x)=\frac{1}{\lambda}\varphi(\frac{x}{\sqrt{\lambda}})$  with $\lambda=\lambda(\varphi)$. Then $\varphi_{\lambda}(x)$  satisfies
\eqref{model14}.
\end{lemma}
\begin{proof}
Notice that
$\varphi(x)=\lambda\varphi_{\lambda}(\sqrt{\lambda}x)$.  Then
$$-\lambda^2\Delta \varphi_{\lambda}(z)+\Big(\int_{|z_1|\leq|z|}\lambda^{3}\lambda^{\frac{N-2}{2}}(|z_1|^{2-N}-|z|^{2-N})|\varphi_{\lambda}(z_1)|^2\lambda^{\frac{-N}{2}}dz_1\Big)
\varphi_{\lambda}(z)=\lambda^2\varphi_{\lambda}(z),$$
where $z=\sqrt{\lambda}x.$
Therefore,
$$(-\Delta+U_{\varphi_\lambda}(z))
\varphi_\lambda(z)=\varphi_\lambda(z).$$
This completes our proof.
\end{proof}

In the following part,  we shall first   investigate  the uniqueness
of positive radial solutions of \eqref{model14}, which yields that the uniqueness
of positive radial solutions of \eqref{model5} due to Lemma  \ref{lemma2}.  Assume that $\phi\in  {C}^2(B_{R^*})\bigcap H^1_0(B_{R^*})$ is a positive radial solution of  \eqref{model14}.

Define  a new functional $A_\phi:H^1_0(B_{R^*})\to \mathbb{R}$ by
\begin{equation}
A_\phi(\psi)=\int_{B_{R^*}}|\nabla\psi(x)|^2dx+\int_{B_{R^*}}U_\phi(x)|\psi(x)|^2dx.
\end{equation}
We shall consider the following minimizing problem
$$\Gamma_\phi=\inf\{A_\phi(\psi)|\psi\in H_0^1(B_{R^*}),\ \|\psi\|_{L^2(B_{R^*})}=1\}.$$
Then we have the following lemma.
\begin{lemma}\label{lemma5}
$\Gamma_\phi$ can be achieved by    a nonnegative  radial function  $\hat{\psi}\in H^1_0(B_{R^*})$   and $\Gamma_\phi=1$.
\end{lemma}
\begin{proof}
It is easy to see that  $\Gamma_\phi\geq0$. For any fixed  $\psi\in H_0^1(B_{R^*})$, we
denote the  symmetric decreasing rearrangement of
$\psi$ by $\psi^*$.  In view of Lemma \ref{lemma10}. we assume  that   $U_\phi^\infty:=\lim\limits_{|x|\to R}U_\phi(x)$. Then by  using symmetric  rearrangement inequalities(see
\cite{LEH}),  there holds that $\|\nabla \psi\|_{L^2(B_{R^*})}\geq\|\nabla \psi^*\|_{L^2(B_{R^*})},$ $\| \psi\|_{L^2(B_{R^*})}=\|\psi^*\|_{L^2(B_{R^*})}$
and
\begin{equation}
\begin{array}{lll}
\displaystyle\int_{B_{R^*}}U_\phi(x)|\psi(x)|^2dx&=&\displaystyle\int_{B_{R^*}}U_\phi^\infty|\psi(x)|^2dx
-\int_{B_{R^*}}(U_\phi^\infty-U_\phi(x))|\psi(x)|^2dx\\
&\geq&\displaystyle\int_{B_{R^*}}U_\phi(x)|{\psi^*}(x)|^2dx.\\
\end{array}
\end{equation}

From the above, we can assume that  there exists a sequence of nonnegative radially symmetric   decreasing functions  $\{\psi_n\}_{n\geq1}$
satisfying $A_{\phi}(\psi_n)\rightarrow \Gamma_\phi$ and $\|\psi_n\|_{L^2(B_{R^*})}=1$. Since
$\|\nabla \psi_n\|_{L^2(B_{R^*})}$ is bounded, up to a subsequence,  there exists a nonnegative radial symmetric decreasing function $\hat{\psi}\in H^1_0(B_{R^*})$ such that   $\psi_n\rightharpoonup \hat{\psi}$ weakly in
$H_0^1(B_{R^*})$  and $\psi_n\rightarrow \hat{\psi}$ strongly in
$L^2(B_{R^*})$.   Thus $\|\hat{\psi}\|_{L^2(B_{R^*})}=1$. Moreover, by  weak lower semicontinuity of the norm and Fatou's Lemma , we
have $\liminf\limits_{n\rightarrow \infty}\|\nabla \psi_n\|_2\geq\|\nabla \hat{\psi}\|_2$ and
$$\liminf\limits_{n\rightarrow \infty}\int_{B_{R^*}}U_\phi(x)|\psi_n(x)|^2dx\geq\int_{B_{R^*}}U_\phi(x)|{\hat{\psi}}(x)|^2dx,$$
which  shows that $A_\phi(\hat{\psi})\leq \Gamma_\phi$. This  combined with the fact that $\|\hat{\psi}\|_{L^2(B_{R^*})}=1$, implies that $\Gamma_\phi$ can be achieved by $\hat{\psi}$.

 Next, it suffices to
prove  $\Gamma_\phi=1$.
By using  Lagrange multiplier principle, there exists a real number $\theta$  such that
\begin{equation}\label{model19}
(-\Delta+U_{\phi}(x))\hat{\psi}=\theta\hat{\psi}.
\end{equation}
It is easy to check that $\theta=\Gamma_\phi$.
Multiplying \eqref{model14}and \eqref{model19} by $\hat{\psi}$ and $\phi $, respectively, and integrating by part,  we deduce that  $\Gamma_\phi=1$. This completes our proof.
\end{proof}

In similar spirit of \cite[Lemma 3.3]{wy},  we obtain the  following lemma.
\begin{lemma}\label{lemma9}
The positive  radial solution of
\eqref{model14} is unique.
\end{lemma}
\begin{proof}
We shall argue it by contradiction. Let  $\phi_1$ and $\phi_2\in {C}^2(B_{R^*})\bigcap
H^1_0(B_{R^*})$ be two different positive radial solutions of \eqref{model14}.  Then   $A_{\phi_i}(\phi_i)=\|\phi_i\|^2_2$, and
$\phi_i$ satisfies the following ordinary differential
equation with second order
\begin{equation}\label{model27}
\phi^{''}(r)+\frac{N-1}{r}\phi^{'}(r)=(U_\phi(r)-1)\phi(r),
\end{equation}
 where $i=1,2$.
Set $\psi:=\phi_1-\phi_2\nequiv0$. Then $\psi$ is also a radial function. There are  three
cases to occur.

\textbf{Case 1}. Either $\psi(r)\geq0$ for all $r\in[0, R^*)$ or
$\psi(r)\leq0$ for all $r\in[0,R^*)$.

Without loss of generality, we assume  $\psi(r)\geq0$ for all
$r\in[0,R^*)$. For any $\varphi\in
H^1_0(B_{R^*})$,  there holds
\begin{equation}\label{model26}
A_\phi(\varphi)\geq\|\varphi\|_{L^2(B_{R^*})}^2
\end{equation} due to  Lemma~\ref{lemma5}. This combine with Lemma \ref{lemma10}, yields that
$$\|\phi_1\|_{L^2(B_{R^*})}^2\leq A_{\phi_2}(\phi_1)=A_{\phi_1}(\phi_1)+\int_{B_R}(U_{\phi_2}-U_{\phi_1})|\phi_1(x)|^2<\|\phi_1\|^2_{L^2(B_{R^*})},$$
a contradiction. This case will not happen.

\textbf{Case 2}. There is $R_1\in(0,R^*)$ such that $\psi(R_1)=0$,  and  either
$\psi\gneqq 0$ in $[0,R_1]$ or $\psi\lneqq0$ in $[0,R_1]$.

Without loss of generality, we assume  $\psi(R_1)=0$ and  $\psi\gneqq
0$ in $[0,R_1]$.
Define
$$\tilde{\psi}(x)=\left\{\begin{array}{lll}\psi(x),\  |x|\leq R_1,\\
0,\   |x|\in(R_1,R^*).\end{array}\right.$$
 Since $\phi_1$ and $\phi_2$
satisfy \eqref{model14}, then
 $$[-\Delta +\frac{1}{2}(U_{\phi_1}+U_{\phi_2})]\tilde{\psi}=\tilde{\psi}-\frac{U_{\phi_1}-U_{\phi_2}}{2}(\phi_1+\phi_2),  \  |x|\leq R_1.$$
 Multiplying this by $\tilde{\psi}$ and integrating by part, we have
 $$\frac{1}{2}A_{\phi_1}(\tilde{\psi})+\frac{1}{2}A_{\phi_2}(\tilde{\psi})=\|\tilde{\psi}\|^2_2-\int_{B_{R^*}}\frac{U_{\phi_1}-U_{\phi_2}}{2}(\phi_1+\phi_2)\tilde{\psi}dx<\|\tilde{\psi}\|^2_2,$$
which leads to a contradiction with \eqref{model27}.
This case does not hold.

 \textbf{Case 3}. There exists  $R_2\in[0,R^*)$ such that $\psi\equiv0$ in $[0,R_2]$, and for any $\varepsilon>0$, $\psi$  changes sign in  $(R_2,R_2+\varepsilon)$.

Notice that
$\phi_1(R_2)=\phi_2(R_2), \ \phi_1^{'}(R_2)=\phi_2^{'}(R_2).$
Applying the variation of constants formula to \eqref{model27},  we obtain
\begin{equation}
\displaystyle\phi_1(r)-\phi_2(r) = T(r,\phi_1)-T(r,\phi_2)
\end{equation}
where
\begin{equation}\label{model32}
T(r,\phi_i)=\int_{R_2}^r\frac{s}{N-2}(U_{\phi_i}(s)-1)\phi_i(s)ds+
\frac{\int_{R_2}^r\frac{s^{N-1}}{N-2}(1-U_{\phi_i}(s))\phi_i(s)ds}{r^{N-2}}
\end{equation}
with $i=1,2$.
For any  $r\in({R_2},{R_2}+\varepsilon)$, we
obtain the following two estimates.
\begin{equation}\label{model28}
\begin{array}{lll}
\displaystyle U_{\phi_i}(r)&=&\displaystyle\int^r_0 |\mathbb{S}^{N-1}|s^{N-1}(s^{2-N}-r^{2-N})|\phi_{i}(s)|^2ds\\
&\leq& Cr^2,
\end{array}
\end{equation}
and
\begin{equation}\label{model29}
\begin{array}{lll}
\displaystyle| U_{\phi_1}(r)-U_{\phi_2}(r)|&\leq&\displaystyle\int^r_0 |\mathbb{S}^{N-1}|s^{N-1}(s^{2-N}-r^{2-N})|\phi_1(s)-\phi_2(s)||\phi_1(s)+\phi_2(s)|ds\\
&=&\displaystyle\int^r_{R_2} |\mathbb{S}^{N-1}|s^{N-1}(s^{2-N}-r^{2-N})|\phi_1(s)-\phi_2(s)||\phi_1(s)+\phi_2(s)|ds\\
&\leq& C \sup\limits_{s\in({R_2},{R_2}+\varepsilon)}|\phi_1(s)-\phi_2(s)|r^2.
\end{array}
\end{equation}
Then for any $r\in(R_2,R_2+\varepsilon)$ with $\varepsilon$  sufficiently small, we conclude from \eqref{model32},  \eqref{model28} and \eqref{model29} that there exists $0<C_\varepsilon<\frac{1}{2}$ small enough  such that
\begin{equation}
\begin{array}{lll}
\displaystyle |\phi_1(r)-\phi_2(r)|
&\leq&\displaystyle\Big|\int^r_{R_2}\frac{s}{N-2}[U_{\phi_1}(s)(\phi_1(s)-\phi_2(s))+(U_{\phi_1}(s)-U_{\phi_2}(s))\phi_2(s)+(\phi_2(s)-\phi_1(s))]ds\Big|\\
\quad &+&\displaystyle\frac{1}{{r^{N-2}}}{\Big|\int_{R_2}^r\frac{s^{N-1}}{N-2}
[(\phi_1(s)-\phi_2(s))+U_{\phi_2}(s)(\phi_2(s)-\phi_1(s))+(U_{\phi_2}(s)-U_{\phi_1}(s))\phi_1(s)]ds\Big|}\\
&\leq&
\displaystyle C_{\varepsilon}\sup\limits_{s\in({R_2},{R_2}+\varepsilon)}|\phi_1(s)-\phi_2(s)|\\
\end{array}
\end{equation}
This implies a contradiction with  $\psi\nequiv0$ in $r\in({R_2},{R_2}+\varepsilon)$. This case is not true.

From the above arguments, we have $\phi_1\equiv\phi_2$.
 The proof is completed.
\end{proof}

Now we are ready to prove the uniqueness of positive ground state solutions of \eqref{model5}.
\begin{thm}
Assume that   $N=3,4,5,6$. Then the    positive  ground state solution  of  \eqref{model5} is uniquely determined.
\end{thm}
\begin{proof}
We shall argue  by contradiction. Suppose on the contrary that $\varphi_1$ and $\varphi_2\in H_0^1(B_{R})\bigcap C^2(B_R)$ are two distinct positive ground state solutions of \eqref{model5}.  In view of  \eqref{model13} and \eqref{model12}, we shall finish the proof  by  distinguishing   two cases: $\lambda(\varphi_1)=\lambda(\varphi_2)$ (\emph{the first case})
$\lambda(\varphi_1)\neq \lambda(\varphi_2)$ (\emph{the second case}).

 \textbf{The first Case}.  Since $\lambda(\varphi_1)=\lambda(\varphi_2)$,  by applying Lemma \ref{lemma2},
there exist two distinct positive radial   solutions $\phi_1$ and $\phi_2\in   H_0^1(B_{\sqrt{\lambda(\varphi_1)}R})\bigcap C^2(B_{\sqrt{\lambda(\varphi_1)}R})$ of \eqref{model14}. This implies a contradiction with Lemma \ref{lemma9}. The first case will not happen.

 \textbf{The second  Case}. Note that $\lambda(\varphi_1)\neq \lambda(\varphi_2)$.  Without loss of generality, we assume that $\lambda(\varphi_1)> \lambda(\varphi_2)$. Let $\tilde{\lambda}=\frac{\lambda(\varphi_1)}{\lambda(\varphi_2)}$ and $\tilde{\varphi}_2(x)=\frac{1}{\tilde{\lambda}}\varphi_2(\frac{x}{\sqrt{\tilde{\lambda}}})$ with $x\in B_{\sqrt{\tilde{\lambda}}R}$. Then a direct calculation yields that $\tilde{\varphi}_2$ satisfies the following equation
 \begin{equation}\label{model33}
 (-\Delta+U_{\tilde{\varphi}_2}){\tilde{\varphi}_2}(x)=\lambda(\varphi_1){\tilde{\varphi}_2}(x),\   x\in B_{\sqrt{\tilde{\lambda}}R}.
\end{equation}
This together with \eqref{model13}, shows that both  $\varphi_1$ and $\tilde{\varphi}_2$ satisfy
\begin{equation}\label{model34}
 (-\Delta+U_{\varphi})\varphi(x)=\lambda(\varphi_1)\varphi(x),\   x\in B_R,
\end{equation}
i.e.
\begin{equation}\label{model35}
\varphi^{''}(r)+\frac{N-1}{r}\varphi^{'}(r)=(U_\varphi(r)-\lambda(\varphi_1))\varphi(r).
\end{equation}

Note that $\varphi_1'(0)=\tilde{\varphi}_2'(0)=0.$ If $\varphi_1(0)>\tilde{\varphi}_2(0),$  by integrating \eqref{model35}, we obtain
$$\int_{0}^r [s^{N-1}\varphi_1''(s)\tilde{\varphi}_2(s)+(N-1)s^{N-2}\varphi_1'(s)\tilde{\varphi}_2(s)]ds
=\int_0^rs^{N-1}(U_{\varphi_1}(s)-\lambda(\varphi_1))\varphi_1(s)\tilde{\varphi}_2(s)ds.$$
So using the fact $\varphi_1'(0)=\tilde{\varphi}_2'(0)=0,$ there holds that
$$r^{N-1}\varphi_1'(r)\tilde{\varphi}_2(r)-\int_{0}^r s^{N-1}\varphi_1'(s)\tilde{\varphi}_2'(s)ds=\int_0^rs^{N-1}(U_{\varphi_1}(s)-\lambda(\varphi_1))\varphi_1(s)\tilde{\varphi}_2(s)ds.$$
Similar arguments can lead to
$$r^{N-1}\tilde{\varphi}_2'(r)\varphi_1(r)\int_{0}^r s^{N-1}\varphi_1'(s)\tilde{\varphi}_2'(s)ds=\int_0^rs^{N-1}(U_{\tilde{\varphi}_2}(s)-\lambda(\varphi_1))\varphi_1(s)\tilde{\varphi}_2(s)ds.$$
Thus
\begin{equation}\label{model36}
(\frac{\varphi_1(r)}{\tilde{\varphi}_2(r)})'=\frac{1}{r^{N-1}{\tilde{\varphi}}_2^2(r)}
\int_0^rs^{N-1}(U_{\varphi_1}(s)-U_{\tilde{\varphi}_2}(s))\varphi_1(s)\tilde{\varphi}_2(s)ds.
\end{equation}
Since $\varphi_1(0)>\tilde{\varphi}_2(0),$ this implies that $\varphi_1(r)>\tilde{\varphi}_2(r)$ on $[0,t)$ with small $t>0$. Notice that $(\frac{\varphi_1(r)}{\tilde{\varphi}_2(r)})'>0$ on $[0,t)$, which can yields that $\varphi_1(r)>\tilde{\varphi}_2(r)$ on $B_R.$ This  leads to a contradiction with
the definition of $\tilde{\varphi}_2.$

If $\varphi_1(0)<\tilde{\varphi}_2(0),$  by using \eqref{model36},  we obtain $\varphi_1(r)<\tilde{\varphi}_2(r)$ on $B_R.$
Since $\varphi_1$ and $\varphi_2$ are  two positive ground state  solutions of \eqref{model13}, so
\begin{equation}\label{model37}
I_R(\varphi_1)=\frac{1}{4}\int_{B_R}\int_{B_R}G(x,y)|\varphi_1(y)|^2|\varphi_1(x)|^2 dxdy
=\frac{1}{4}\int_{B_R}\int_{B_R}G(x,y)|\varphi_2(y)|^2|\varphi_2(x)|^2 dxdy=I_R(\varphi_2).
\end{equation}
According to Lemma \ref{lemma8}, since $\tilde{\lambda}>1$,  we conclude that when $N=3,4,5,6$, there holds
\begin{equation}
\begin{array}{lll}
\displaystyle\int_{B_R}\int_{B_R}G(x,y)|\varphi_1(y)|^2|\varphi_1(x)|^2 dxdy\\
 <\displaystyle\int_{B_R}\int_{B_R}G(x,y)|\tilde{\varphi}_2(y)|^2|\tilde{\varphi}_2(x)|^2 dxdy\\
=\displaystyle\int_{B_R}\int_{B_R}\frac{|\tilde{\varphi}_2(y)|^2|\tilde{\varphi}_2(x)|^2}{|y-x|^{N-2}}dxdy-R^{2-N}\int_{B_R}\int_{B_R}|\tilde{\varphi}_2(y)|^2|\tilde{\varphi}_2(x)|^2 dxdy\\
=\displaystyle{\tilde{\lambda}}^{\frac{N}{2}-3}\int_{B_{\frac{R}{\sqrt{\tilde{\lambda}}}}}\int_{B_{\frac{R}{\sqrt{\tilde{\lambda}}}}}
\frac{|\varphi_2(y)|^2|\varphi_2(x)|^2}{|y-x|^{N-2}}dxdy-{\tilde{\lambda}}^{N-4}R^{2-N}\int_{B_{\frac{R}{\sqrt{\tilde{\lambda}}}}}
\int_{B_{\frac{R}{\sqrt{\tilde{\lambda}}}}}|\varphi_2(y)|^2|\varphi_2(x)|^2 dxdy\\
=\displaystyle\int_{B_{\frac{R}{\sqrt{\tilde{\lambda}}}}}\int_{B_{\frac{R}{\sqrt{\tilde{\lambda}}}}}
\frac{|\varphi_2(y)|^2|\varphi_2(x)|^2}{|y-x|^{N-2}}dxdy-R^{2-N}\int_{B_{\frac{R}{\sqrt{\tilde{\lambda}}}}}
\int_{B_{\frac{R}{\sqrt{\tilde{\lambda}}}}}|\varphi_2(y)|^2|\varphi_2(x)|^2 dxdy\\
<\displaystyle\int_{B_R}\int_{B_R}G(x,y)|\varphi_2(y)|^2|\varphi_2(x)|^2 dxdy.
\end{array}
\end{equation}
This implies a contradiction with \eqref{model37}.

If $\varphi_1(0)=\tilde{\varphi}_2(0),$ by applying similar arguments as \textbf{Case 3} in Lemma \ref{lemma9}, we  deduce from \eqref{model35} that there exists $\delta_1>0$ small enough and $0<C_{\delta_1}<\frac{1}{2}$   such that
\begin{equation}
\begin{array}{lll}
\displaystyle |\varphi_1(r)-\tilde{\varphi}_2(r)|
&\leq&\displaystyle\Big|\int^r_{0}\frac{s}{N-2}[U_{\varphi_1}(s)(\varphi_1(s)-\tilde{\varphi}_2(s))+
(U_{\varphi_1}(s)-U_{\tilde{\varphi}_2}(s))\tilde{\varphi}_2(s)+\lambda(\varphi_1)(\tilde{\varphi}_2(s)-\varphi_1(s))]ds\Big|\\
\quad &+&\displaystyle\frac{1}{r^{N-2}}{\Big|\int_{0}^r\frac{s^{N-1}}{N-2}
[\lambda(\varphi_1)(\varphi_1(s)-\tilde{\varphi}_2(s))+U_{\tilde{\varphi}_2}(s)(\tilde{\varphi}_2(s)-\varphi_1(s))
+(U_{\tilde{\varphi}_2}(s)-U_{\varphi_1}(s))\varphi_1(s)]ds\Big|}\\
&\leq&
\displaystyle C_{\delta_1}\sup\limits_{s\in({0},{0}+{\delta_1})}|\varphi_1(s)-\tilde{\varphi}_2(s)|.
\end{array}
\end{equation}
Thus, $\varphi_1(r)=\tilde{\varphi}_2(r)$ on $[0,\delta_1]$. By applying the iteration arguments with starting point $\delta_1$, we can finally obtain that
that $\varphi_1(r)=\tilde{\varphi}_2(r)$ on $B_R.$ This implies  a contradiction with
the definition of $\tilde{\varphi}_2.$
From the above arguments, we see that the second  case is not valid.

Therefore,   we complete the proof.
\end{proof}

\section{Convergence}
In this section, we shall show the convergence of the unique positive  ground state solution $\phi_R$ of \eqref{model5} as $R\to\infty.$

 First,  consider the following Choquard equation in full space $\mathbb{R}^N$
\begin{equation}\label{model21}
-\Delta u+u=\Big(\int_{\mathbb{R}^N}\frac{|u^2(y)|}{|x-y|^{N-2}}dy\Big)u,\ \ x\in \mathbb{R}^N.
\end{equation}
The corresponding energy functional  $I_\infty:H^1(\mathbb{R}^N)\to \mathbb{R}$ associated to \eqref{model21} is
\begin{equation}\label{model22}
I_\infty(u)=\frac{1}{2}\int_{\mathbb{R}^N}(|\nabla u|^2+|u|^2)dx-
\frac{1}{4}\int_{\mathbb{R}^N}\int_{\mathbb{R}^N}\frac{|u(y)|^2|u(x)|^2}{|x-y|^{N-2}}dxdy,
\end{equation}
Let $\mathcal{N}_\infty=\{u\in H^1(\mathbb{R}^N)\backslash{\{0\}}:\langle I_\infty'(u),u\rangle=0\}$ and $c_\infty=\inf\limits_{u\in \mathcal{N}_\infty}I_\infty(u).$
As we know, $c_\infty$ can be achieved by  a unique positive radial solution $\phi_\infty$ of \eqref{model22} (see \cite{wy}). Notice that
$c_R=I_R(\phi_R)$.
Then we have the following lemma.
\begin{prop}
There holds $\lim\limits_{R\to \infty}c_R=c_\infty.$
\end{prop}
\begin{proof}
We first prove that $c_R\geq c_\infty.$ Indeed, by similar  arguments as in Lemma \ref{lemma7}, there exists a unique $t_R>0$ such that $t_R\phi_R\in \mathcal{N}_\infty.$ Since
$\langle I'_R(\phi_R),\phi_R\rangle=0$, we see that $\langle I'_\infty(\phi_R),\phi_R\rangle<0$, which implies that
$t_R\in(0,1)$.
Hence
 $$c_\infty\leq I_\infty(t_R\phi_R)=\frac{1}{4}t_R^2\|\phi_R\|^2\leq
\frac{1}{4}\|\phi_R\|^2 = I_R(\phi_R)=c_R.$$

On the other hand, we show $\lim\limits_{R\to\infty}c_R\leq c_\infty.$ Let $\eta_R\in C^\infty(\mathbb{R}^N)$ be a radial function such that $\eta_R=1$ in
$B_{\frac{R}{2}}$, $\eta_R\in(0,1)$ in $B_R\backslash B_{\frac{R}{2}}$,
$\eta_R=0$ in $\mathbb{R}^N\backslash B_R$ and $|\nabla \eta_R|\leq \frac{2}{R}.$
Define $\Psi_R=\eta_R\phi_\infty.$
In fact,  by standard arguments, we can deduce that
$\Psi_R\to \phi_\infty$ in $H^1(\mathbb{R}^N)$ as $R\to\infty.$  In addition,  by using Lemma \ref{lemma8},
\begin{equation}\label{model30}
\begin{split}
\lim\limits_{R\to\infty}\int_{B_R}\int_{B_R}G(x,y)|\Psi_R(y)|^2|\Psi_R(x)|^2dxdy
=&\lim\limits_{R\to\infty}\Big[\int_{B_R}\int_{B_R}|x-y|^{2-N}|\Psi_R(y)|^2|\Psi_R(x)|^2dxdy+R^{2-N}\|\Psi_R\|_{L^2}^4\Big]\\
=& \int_{\mathbb{R}^N}\int_{\mathbb{R}^N}|x-y|^{2-N}|\phi_\infty(y)|^2|\phi_\infty(x)|^2dxdy:=\mathbb{E}(\phi_\infty).
\end{split}
\end{equation}
Then $\lim\limits_{R\to\infty}I_R(\Psi_R)=I_\infty(\phi_\infty)=c_\infty.$ By using Lemma \ref{lemma7},
  there exists a unique $s_R>0$ such that $s_R\Psi_R\in \mathcal{N}_R$. Moreover,   since $\phi_\infty\in \mathcal{N}_\infty,$  we obtain $$s_R=\left(\frac{\|\Psi_R\|^2}{\mathbb{D}(\Psi_R)}\right)^\frac{1}{2}\to \left(\frac{\|\phi_\infty\|^2}{\mathbb{E}(\phi_\infty)}\right)^\frac{1}{2} =1$$ as $R\to \infty.$ Thus
$$\lim\limits_{R\to\infty}c_R\leq \lim\limits_{R\to\infty}I_R(s_R\Psi_R)=\lim\limits_{R\to\infty}I_R(\Psi_R)=I_\infty(\phi_\infty)=c_\infty.$$
Therefore $c_R\to c_\infty$ as $R\to\infty.$
\end{proof}
\begin{thm}
Assume that   $N=3,4,5,6$. Then the    positive  ground state solution  of  \eqref{model5}  converges  to the unique positive ground state solution of  \eqref{model1}  as    $R\to\infty$.
\end{thm}

\begin{proof}
It suffices to prove that $\phi_R\to \phi_\infty $ in $H^1(\mathbb{R}^N)$ as $R\to \infty$. Since $c_R\to c_\infty$ as $R\to\infty,$  we conclude that  $\{\phi_R\}$ is uniformly bounded with respect to $R$. Then there exists a nonnegative radially symmetric decreasing function $\psi\in H^1(\mathbb{R}^N)$ such that   $\phi_R\rightharpoonup \psi$ weakly in
$H^1(\mathbb{R}^N)$  and $\phi_R\rightarrow \psi$ strongly in
$L^s(\mathbb{R}^N)$ with $s\in(2,6)$ due to the radial symmetry of $\phi_R$.
In view of \eqref{model30},  $$0\neq\lim\limits_{R\to \infty}\int_{B_R}\int_{B_R}G(x,y)|\phi_R(y)|^2|\phi_R(x)|^2dxdy=\int_{\mathbb{R}^N}\int_{\mathbb{R}^N}|x-y|^{2-N}|\psi(y)|^2|\psi(x)|^2dxdy.$$
Then we have  $\psi\neq   0$. Furthermore, for any $v\in C_0^\infty(\mathbb{R}^N)$,
$$0=\lim\limits_{R\to \infty}\langle I'_R(\phi_R),v\rangle=\langle I'_\infty(\psi) ,v\rangle.$$
The uniqueness of the positive radial  solution of  \eqref{model21} yields that $\psi=\phi_\infty.$
This combined with the fact that $\lim\limits_{R\to \infty}c_R= c_\infty$,  implies that $\phi_R\to \phi_\infty $ strongly in $H^1(\mathbb{R}^N)$ as $R\to \infty$.
The proof is completed.
\end{proof}

\end{document}